\documentclass{smfart}
\usepackage{amsmath, amssymb, amsthm, amscd,smfenum,xspace,mathrsfs,url,upref,verbatim}
\usepackage[utf8]{inputenc} 
\usepackage[T1]{fontenc}
\usepackage{dsfont}

\usepackage{stmaryrd}

\usepackage[all,cmtip]{xy}
\usepackage{xfrac} 

\usepackage{ textcomp }

\let\hat=\widehat
\let\tilde=\widetilde
\numberwithin{equation}{subsection}

\newtheorem{theorem}[equation]{Théorème} 
\newtheorem{proposition}[equation]{Proposition}
\newtheorem{lemme}[equation]{Lemme}

\theoremstyle{remark}

\newtheorem{definition}[equation]{Définition}

\DeclareMathOperator{\ord}{ord}
\DeclareMathOperator{\Max}{Max}
\DeclareMathOperator{\id}{id}

\DeclareMathOperator{\Spec}{Spec}
\DeclareMathOperator{\Sym}{Sym}
\DeclareMathOperator{\Proj}{Proj}
\DeclareMathOperator{\red}{red}
\def\cartesien{\ar@{}[rd]|{\square}}
\DeclareMathOperator{\sw}{sw}
\DeclareMathOperator{\rsw}{rsw}

\DeclareMathOperator{\Min}{Min}

\DeclareMathOperator{\Ext}{Ext}

\title[La construction d'Abbes et Saito pour les connexions]{La construction d'Abbes et Saito pour les connexions méromorphes: aspect formel en dimension 1 
}

\author[J.-B.~Teyssier]{Jean-Baptiste Teyssier}
\date{}
\curraddr{Centre de Mathématiques Laurent Schwartz, École polytechnique,
91128 Palaiseau cedex, France}
\email{jean-baptiste.teyssier@math.polytechnique.fr}

\usepackage[french,english]{babel}
\begin{document}

\maketitle
\begin{abstract} 
Dans \cite{AS}, Abbes et Saito définissent une mesure géométrique de la ramification sauvage des faisceaux $\ell$-adiques sur le point générique d'un trait complet d'égale caractéristique $p$, avec $p\neq \ell$. En adaptant leur construction aux modules différentiels en égale caractéristique nulle, on démontre pour un tel module $\mathcal{M}$ une formule qui exprime cet invariant géométrique en terme des formes différentielles intervenant dans la décomposition de Levelt-Turrittin de $\mathcal{M}$. 
\end{abstract}

\section*{Introduction}
Inspirés par l'analyse micro-locale de Kashiwara et Schapira \cite{KS}, Abbes et Saito \cite{AS} associent à tout faisceau $\ell$-adique $\mathcal{F}$ sur le point générique d'un trait complet $S$ d'égale caractéristique $p$, et ce pour tout nombre rationnel $r>0$, un nombre fini de formes différentielles tordues, généralisant ainsi en rang supérieur la construction du conducteur de Swan raffiné pour les caractères d'Artin-Schreier-Witt développée par Kato \cite{Kato}. 
Ces formes sont obtenues comme support d'un faisceau construit à partir de $\mathcal{F}$  via des manipulations de nature géométrique, du foncteur des cycles proches et de la transformation de Fourier $\ell$-adique. En particulier, ces ingrédients sont disponibles dans le cadre des $\mathcal{D}$-modules. \\ \indent 
Motivé par les analogies entre l'irrégularité des $\mathcal{D}$-modules complexes et la ramification sauvage des faisceaux $\ell$-adiques en caractéristique positive, on peut donc se demander ce que donne cette construction en égale caractéristique nulle lorsqu'on remplace $\mathcal{F}$ par un module différentiel $\mathcal{M}$.  \\ \indent
Lorsque $r>1$, on démontre dans ce travail une formule explicite \ref{théorème} reliant la construction d'Abbes et Saito appliquée à $\mathcal{M}$ aux polynômes de Laurent de degré $\leq r-1$ intervenant dans la décomposition de Levelt-Turrittin de $\mathcal{M}$. En particulier, \ref{théorème} confirme que les invariants de la ramification sauvage construits par Abbes et Saito "sont les bons". Dans le cas où $\mathcal{M}$ est de pente unique $r'>0$, le support du $\mathcal{D}$-module obtenu pour $r=r'+1$ est ponctuel et correspond à l'ensemble  des coefficients dominants (à multiplication par $1-r$ près) des polynômes de Laurent de degré $r^{\prime}$ attachés à $\mathcal{M}$. Il s'agit de l'analogue pour les modules différentiels du théorème \cite[9.15]{AS} d'Abbes et Saito\footnote{Maintenant prouvé dans le cas à corps résiduel non nécessairement parfait par Saito. Voir \cite{2013arXiv1301.4632S}.}. \\ \indent
Si $X$ est une variété complexe lisse, $Y$ une hypersurface lisse de $X$ et $\mathcal{M}$ un $\mathcal{D}_{X}$-module holonome, Yves Laurent \cite{Laurent} \cite{LaurentPolygone} définit pour tout nombre rationnel $r>1$ un cycle lagrangien $\sigma_{r}(\mathcal{M})$ de $T^{\ast}T_{Y}^{\ast}X$ qui fournit une mesure géométrique de l'irrégularité de $\mathcal{M}$ le long de $Y$. C'est la notion de \textit{cycle micro-caractéristique}. Dans le cas d'un trait, le support de $\sigma_{r}(\mathcal{M})$ est déterminé par l'ensemble des coefficients dominants des polynômes de degré $1/(r-1)$ attachés à $\mathcal{M}$. Le théorème \ref{théorème} montre en particulier que dans le cas d'un trait sur un corps algébriquement clos, la construction d'Abbes et Saito constitue une variante de la théorie de Laurent. L'auteur ignore si ceci subsiste en dimension supérieure.\\ \indent
Pour démontrer \ref{théorème}, on commence en \ref{Réduction} par réduire le problème au cas où le corps de base est $\mathds{C}$.
Il s'agit d'une manifestation du principe de Lefschetz. La stratégie est alors de se ramener à la situation où $\mathcal{M}$ est donné sous forme décomposée tout en contrôlant la façon dont sont affectés les cycles proches qui interviennent dans la construction d'Abbes et Saito. On conclut alors grâce à divers lemmes d'annulation et à un calcul explicite.\\ \indent
Ce texte est une partie de la thèse de l'auteur effectuée sous la direction de Claude Sabbah. Je le remercie pour avoir partagé avec moi son intuition que la construction d'Abbes et Saito devait être reliée aux parties les plus polaires des formes de Levelt-Turrittin, ainsi que pour m'avoir inculqué avec patience tout ce que je sais des $\mathcal{D}$-modules. Je remercie Ahmed Abbes et Marco Hien pour l'intérêt qu'ils ont porté à ce travail lors de son élaboration. Je remercie aussi le référé pour de nombreuses remarques qui ont contribué à améliorer la lisibilité de ce texte.
\section{Notations}\label{notationetrappel}
\subsection{}
On désigne par $\mathds{K}$ un corps de caractéristique nulle, par $\overline{\mathds{K}}$ une clôture algébrique de $\mathds{K}$, et on note $G_{\mathds{K}}$ le groupe de Galois de $\overline{\mathds{K}}$ sur $\mathds{K}$. Pour une extension quelconque $\mathds{L}$ de $\mathds{K}$, la présence d'un indice $\mathds{L}$ sera synonyme de changement de base à une situation sur $\mathds{L}$. Cet indice sera omis lorsque $\mathds{L}=\mathds{K}$. \\ \indent
Si $X$ est une variété sur $\mathds{K}$ et $P$ un point fermé de $X$, on désignera par $\mathds{K}(P)$ le corps résiduel de $P$. Il s'agit d'une extension finie de $\mathds{K}$.
\subsection{}
Si $S$ est un schéma et si $\mathcal{E}$ est un faisceau quasi-cohérent sur $S$, on note suivant \cite{EGA2} $\mathbf{V}(\mathcal{E})$ pour le spectre de l'algèbre quasi-cohérente $\Sym_{\mathcal{O}_{S}} \mathcal{E}$ et $\mathbf{P}(\mathcal{E})$ pour le $\Proj$ de $\Sym_{\mathcal{O}_{S}} \mathcal{E}$. Les schémas $\mathbf{V}(\mathcal{E})$ et $\mathbf{P}(\mathcal{E})$ seront dans la suite implicitement munis de leur structure de schéma sur $X$.
\subsection{}
On notera $\mathfrak{F}$ la transformation de Fourier sur $\mathds{A}^{1}_{\mathds{K}}$, et pour un point fermé $P$ de $\mathds{A}^{1}_{\mathds{K}}$, on désignera par $\delta_{P}$ le $\mathcal{D}$-module Dirac en $P$. Dans une coordonnée $y$ de $\mathds{A}^{1}_{\mathds{K}}$, le point $P$ correspond à l'orbite sous $G_{\mathds{K}}$ d'un scalaire $c\in \overline{\mathds{K}}$. Si $\mu_{c}(y)$ est le polynôme minimal de $c$ sur $\mathds{K}$, le Dirac $\delta_{P}$ est par définition le module
$
\mathcal{D}_{\mathds{A}^{1}_{\mathds{K}}}/\mathcal{D}_{\mathds{A}^{1}_{\mathds{K}}}\mu_{c}(y)
$.

\section{La construction d'Abbes et Saito} \label{construction}
\subsection{Prologue géométrique}\label{prologue}
On rappelle ici le nécessaire concernant la notion de dilatation. Pour une exposition plus circonstanciée, on pourra se reporter à \cite{AS}. \\ \indent
Soit $f:Y \rightarrow X$ un morphisme de schémas sur $\mathds{K}$, $D$ un sous-schéma fermé de $X$ défini par un faisceau d'idéaux $\mathcal{I}$ et $E$ un sous-schéma fermé de $f^{-1}(D)$ défini par un faisceau d'idéaux $\mathcal{J}$ sur $Y$. Alors 
$\mathcal{I} \cdot \mathcal{O}_{Y}\subset \mathcal{J}$, de sorte qu'on dispose d'un morphisme de $\mathcal{O}_{Y}$-algèbres graduées
\begin{equation} \label{theta}
\begin{array}{c}
\xymatrix{
\theta : f^{\ast} (\oplus_{\mathds{N}}\mathcal{I}^{n})    \ar[r]
&  \oplus_{\mathds{N}}\mathcal{J}^{n}.
}
\end{array}
\end{equation}
Notons $\tilde{Y}_{E}$ (resp. $\tilde{X}_{D}$) l'éclaté de $Y$ le long de $E$ (resp. $D$). Si $\mathfrak{p}\in \tilde{Y}_{E}$, alors $\theta^{-1}(\mathfrak{p})$ détermine un élément de $\tilde{X}_{D}\times_{X}Y= \Proj f^{\ast} (\oplus_{\mathds{N}}\mathcal{I}^{n})$ si et seulement si $\mathfrak{p}$ est dans l'un des ouverts $D_{+}(\theta(x))$ de $\tilde{Y}_{E}$, avec $x\in \mathcal{I}$ vu dans l'algèbre source comme élément de degré $1$. On en déduit que
$\cup_{x\in \mathcal{I}}D_{+}(\theta(x))$ est le plus grand ouvert de $\tilde{Y}_{E}$, noté $Y_{(D)}$ sur lequel $\theta$ induit un morphisme de schémas $Y_{(D)} \rightarrow \tilde{X}_{D}\times_{X}Y$.
\begin{definition}
On appelle $Y_{(D)}$ \textit{la dilatation de $Y$ en $E$ par rapport à $D$}.
\end{definition} 
Soit $f:Y\rightarrow X$ un morphisme séparé de $\mathds{K}$-schémas localement noethériens et $g:X \rightarrow Y$ une section de $f$. Le morphisme $g$ est alors une immersion fermée. Soit $D$ un sous-schéma fermé de $X$, de complémentaire $U$ et $i:D \longrightarrow X$ l'injection canonique. Notons encore $Y_{(D)}$ le dilaté de $Y$ en $g(D)$ par rapport à $D$. Si $E_{[D]}$ désigne le diviseur exceptionnel de $\tilde{Y}_{g(D)}$ et $E_{(D)}$ l'intersection de $E_{[D]}$ avec l'ouvert $Y_{(D)}$, on dispose d'après \cite[2.10]{AS} du diagramme à carré cartésiens
\[\xymatrix{ 
E_{(D)} \ar[r]  \ar[d] 
& Y_{(D)} \ar[d]
& Y\times_{X}U \ar[d] \ar[l] \\  
D \ar[r]
& X     
&U\ar[l] 
}
\]
Supposons de plus que $D$ est un diviseur de Cartier. Alors par \cite[21.2.12]{EGA4}, $i$ est une immersion régulière donc \cite[16.9.13]{EGA4} assure que la suite des faisceaux conormaux pour $D\overset{i}{\longrightarrow} X \overset{g}{\longrightarrow} Y$
\[
\xymatrix{ 
0 \ar[r]
&i^{\ast}\mathcal{N}_{X/Y}^{\vee} \ar[r]
& \mathcal{N}_{D/Y}^{\vee} \ar[r]
&\mathcal{N}_{D/X}^{\vee} \ar[r]
& 0
}
\]
est exacte. Si $\mathcal{I}$ (resp. $\mathcal{J}$) désigne le faisceau d'idéaux de $D$ dans $X$ (resp. de $g(D)$ dans $Y$), cette suite s'explicite en 
\begin{equation} \label{conormaux}
\begin{array}{c}
\xymatrix{ 
0 \ar[r]
&i^{\ast}\mathcal{N}_{X/Y}^{\vee}  \ar[r]
&\mathcal{J}/\mathcal{J}^{2} \ar[r]^-{g^{\sharp}}
&\mathcal{I}/\mathcal{I}^{2} \ar[r]
& 0  
}.
\end{array}
\end{equation}
Puisque $g$ est une section de  $f$, $f^{\sharp}$ fournit un scindage 
\begin{equation} \label{scindage}
\begin{array}{c}
\xymatrix{ 
\mathcal{J}/\mathcal{J}^{2}  \ar[r]^-{\sim}
& i^{\ast}\mathcal{N}_{X/Y}^{\vee}  \oplus \mathcal{I}/\mathcal{I}^{2}
}.
\end{array}
\end{equation} \indent
Supposons de plus que $g$ est une immersion régulière. Alors, $g \circ i$ est aussi régulière et on a suivant \cite[16.9.3]{EGA4} une identification canonique $\Sym \mathcal{J}/\mathcal{J}^{2} \overset{\sim}{\longrightarrow} \oplus_{\mathds{N}} \mathcal{J}^{n}/\mathcal{J}^{n+1}$, d'où une identification $
E_{[D]}\overset{\sim}{\longrightarrow}\mathbf{P}(\mathcal{J}/\mathcal{J}^{2})$. \\ \indent
Or $\mathcal{I}/\mathcal{I}^{2}\simeq i^{\ast}\mathcal{O}_{X}(-D):=\mathcal{O}_{D}(-D)$ est un fibré en droite sur $D$, d'où on déduit à l'aide de \eqref{scindage} une identification 
\begin{equation}\label{iden}
E_{[D]}\overset{\sim}{\longrightarrow}\mathbf{P}(\mathcal{J}/\mathcal{J}^{2}\otimes \mathcal{O}_{D}(D))\simeq \mathbf{P}(i^{\ast}\mathcal{N}_{X/Y}^{\vee}  \otimes \mathcal{O}_{D}(D) \oplus \mathcal{O}_{D})
\end{equation}
Soit $U$ un ouvert affine de $X$ sur lequel $D$ est défini par une fonction $t$. Avec les notations de \eqref{theta}, on a  
$$E_{(D)|U}=E_{[D]|U}\cap D_{+}(\theta(t))=D_{+}([f^{\sharp}(t)])$$
 avec $[f^{\sharp}(t)] \in \Sym(\mathcal{J}/\mathcal{J}^{2})$ de degré~$1$. Donc à travers l'identification \eqref{scindage}, $\mathfrak{p}\in E_{[D]|U}$ définit un élément de $E_{(D)|U}$ si et seulement si  $\mathfrak{p}$ ne contient pas $0 \oplus [t]\in \Sym (i^{\ast}\mathcal{N}_{X/Y}^{\vee}  \oplus (t)/(t^{2}))$ vu en degré~$1$, soit encore que $\mathfrak{p}$ est d'intersection nulle avec le facteur $(t)/(t^{2})$ placé en degré~$1$. \\ \indent
A travers l'identification \eqref{iden}, le schéma $E_{(D)}$ correspond donc aux idéaux $\mathfrak{p}\in \mathbf{P}(i^{\ast}\mathcal{N}_{X/Y}^{\vee}  \otimes \mathcal{O}_{D}(D) \oplus \mathcal{O}_{D})$ ne rencontrant pas le facteur $\mathcal{O}_{D}$ placé en degré 1. C'est donc selon \cite[8.4.1]{EGA2} le fibré vectoriel $\mathbf{V}(i^{\ast}\mathcal{N}_{X/Y}^{\vee}  \otimes \mathcal{O}_{D}(D))$. Du fait de l'identification $\mathcal{N}_{X/Y}^{\vee} \simeq g^{\ast}(\Omega^{1}_{Y/X})$, on a ainsi suivant \cite[3.5]{AS}
\begin{proposition}[\small{interprétation différentielle de la fibre spéciale du dilaté}]\label{interprétation différentielle}\noindent
Avec les notations de \ref{prologue}, si on suppose que $D$ est un diviseur de Cartier et que $g$ est une immersion régulière, alors on a une identification
\[
\xymatrix{ 
E_{(D)} \ar[r]^-{\sim}
& \mathbf{V}((g\circ i)^{\ast}\Omega^{1}_{Y/X} \otimes \mathcal{O}_{D}(D))
}.
\]
\end{proposition}

\subsection{Enoncé du théorème}
Soit $S$ un trait complet de corps résiduel $\mathds{K}$. Le choix d'une uniformisante $x$ de $S$ induit une identification $S\simeq \text{Spec }\mathds{K}\llbracket x \rrbracket$. 
Soient $n \geq 1$ et $k \geq 1$ des entiers. On pose $r=k/n$, $t=x^{1/n}$ et on note $D_{k}$ le diviseur de degré $k$ de $S_{n}=  \text{Spec }\mathds{K}\llbracket t \rrbracket$. Soient $s_{n}$ le point fermé de $S_{n}$, $\eta_{n}$ son point générique et $\gamma_{n}:S_{n}\rightarrow S$ le morphisme d'élévation à la puissance $n$. Soit $S_{1,n}$ le complété de $S \times S_{n}$ en l'origine. Le graphe de $\gamma_{n}$ induit une immersion fermée $\Gamma_{n} : S_{n} \rightarrow S_{1,n}$. Pour la structure de $S_{n}$-schéma sur $S_{1,n}$ donnée par la seconde projection, on définit $S_{1,n}(D_{k})$ comme le dilaté de $S_{1,n}$ en $\Gamma_{n}(D_{k})$ relativement à $D_{k}$. On en déduit suivant \ref{prologue} le diagramme commutatif à carrés cartésiens
\begin{equation} \label{construction AS}
\begin{array}{c}
\xymatrix{ 
T_{D_{k}} \ar[r]^-{i_{n,k}}  \ar[d] 
& S_{1,n}(D_{k}) \ar[d]_{\pi}
& S \times \eta_{n}\ar[d] \ar[l]_-{j_{n,k}} \\  
D_{k} \ar[r]^{i_{k}}
& S_{n}     
&\eta_{n} \ar[l] 
}
\end{array}
\end{equation} 
avec une identification canonique
\begin{equation*} 
\begin{array}{c}
\xymatrix{ 
T_{D_{k}} \ar[r]^-{\sim} &\mathbf{V}((\Gamma_{n}\circ i_{k})^{\ast}\Omega^{1}_{S_{1,n}/S_{n}}\otimes \mathcal{O}_{D_{k}}(D_{k}))
}.
\end{array}
\end{equation*} 
Concrètement, $\Gamma_{n}(D_{k})$ est le sous-ensemble algébrique de $S_{1,n}$ donné par l'idéal $\mathcal{J}=\left(x-t^{n}, t^{k}\right)$. Le choix des variables $y_{0}=x-t^{n}$ et $y_{1}=t^{k}$ placées en degré $1$ fournit une présentation de l'algèbre éclatée de $S_{1,n}$ en $\mathcal{J}$, soit encore un plongement du schéma associé dans  $S_{1,n} \times \mathds{P}^{1}$. Suivant \ref{prologue}, le dilaté $S_{1,n}(D_{k})$ en est l'ouvert affine $y_{1} \neq 0$, donné dans $S_{1,n} \times \mathds{A}^{1}$ par l'équation $x-t^{n}-t^{k}y=0$, où l'on a posé $y=y_{0}/y_{1}$. D'autre part, si on note $\mathcal{O}_{s_{n}}(D_{k})$ la restriction à $s_n$ de $\mathcal{O}_{D_{k}}(D_{k})$, le choix des coordonnées $x$ et $t$ fournit les identifications  
$$
\Omega_{S_{1,n}/S_{n}}^{1}\simeq \mathds{K}\llbracket x,t \rrbracket \cdot dx \text{\quad et \quad}
\mathcal{O}_{s_{n}}(D_{k})\simeq  \mathds{K}\cdot\frac{1}{x^{r}}\hspace{0.5em},
$$
où $x^{r}$ désigne par convention $t^{k}$. Notons $T_{r}$ le réduit\footnote{Dans \cite{AS}, le foncteur des cycles proches est à valeur dans la catégorie dérivée des faisceaux sur $T_{D_{k},\text{ét}}$. Le fibré $T_{D_{k}}$ est un schéma sur $D_{k}$, non réduit en général. Par invariance du site étale par homéomorphisme universel \cite[Exp VIII]{SGA4}, on peut tout aussi bien se placer sur le réduit $T_{D_{k}}^{\red}$ qui est un schéma sur le point fermé $s_{n}$ de  $S_{n}$. C'est le point de vue qui doit être adopté lorsqu'on considère les cycles proches pour les $\mathcal{D}$-modules.} de $T_{D_{k}}$, et relions $\frac{dx}{x^{r}}$ au choix de la coordonnée $y$ sur $S_{1,n}(D_{k})$. Dans la situation présente, la suite exacte \eqref{conormaux} s'explicite en
\begin{equation*} 
\begin{array}{c}
\xymatrix{ 
0 \ar[r]
&(x-t^{n})/((x-t^{n})^{2},t^k(x-t^n)) \ar[r]
&\mathcal{J}/\mathcal{J}^{2} \ar[r]
&(t^{k})/(t^{k})^{2} \ar[r]
& 0
},
\end{array}
\end{equation*}
de sorte que \eqref{scindage} devient $\mathcal{J}/\mathcal{J}^{2}\simeq (y_{0})\oplus (y_{1})$. Via l'isomorphisme $\mathcal{N}_{S_{n}/S_{1,n}}^{\vee} \simeq \Gamma_{n}^{\ast}\Omega^{1}_{S_{1,n}/S_{n}}$, la coordonnée $y_{0}$ correspond à la classe de la forme différentielle $d(x-t^{n})$, soit encore la classe de $dx$. On en déduit que, vu dans 
\begin{equation}\label{la fibre spéciale}
T_{r}=\mathbf{V}((\Gamma_{n}\circ i^{\red}_{k})^{\ast}\Omega^{1}_{S_{1,n}/S_{n}}\otimes \mathcal{O}_{s_{n}}(D_{k})),
\end{equation} 
la coordonnée $y=y_{0}/y_{1}$ correspond exactement à $\frac{dx}{x^{r}}$. C'est par rapport à cette coordonnée privilégiée de la droite $T_{r}$ que se feront tous les calculs. \\ \indent
Soit $\mathcal{M}$ un $\mathds{K}((x))$-module différentiel. Le protagoniste de cet article est le $\mathcal{D}$-module sur $S_{1,n}(D_{k})$ 
$$
H_{n,k}\left( \mathcal{M} \right):=j_{k,n+}\mathcal{H}om(p^{+}_{2}\gamma^{+}_{n}\mathcal{M}, p^{+}_{1}\mathcal{M}),
$$
où $p_{1} :S \times \eta_{n}\longrightarrow S$ et $p_{2}:S \times \eta_{n}\longrightarrow \eta_{n}$ sont les projections canoniques. \\ \indent
On rappelle que le théorème de Levelt-Turrittin \cite{SVDP} assure l'existence d'un entier $m$ et d'une extension galoisienne finie $\mathds{L}$ de $\mathds{K}$ tels que 
\begin{equation} \label{LT}
\mathds{L}((u))\otimes_{\mathds{K}((x))}\mathcal{M} \simeq \displaystyle{\bigoplus_{\omega\in \mathds{L}[\frac{1}{u}]\frac{1}{u}}} \mathcal{E}^{\omega} \otimes  \mathcal{R}_{\omega}
\end{equation}
avec $u=x^{1/m}$, $\mathcal{E}^{\omega}=(\mathds{L}((u)), d+d\omega)$ et $\mathcal{R}_{\omega}$ régulier de rang noté $n_{\omega}$. Le plus petit entier $m$ tel que \eqref{LT} ait lieu est \textit{l'indice de ramification de $\mathcal{M}$}. On le notera $m_{\mathcal{M}}$. \\ \indent 
Posons $\mathds{A}^{\mathds{Q}_{>0}}_{\mathds{K}}=\Spec \mathds{K}[X_r,r\in\mathds{Q}_{>0}]$. Cet espace est muni de projections 
$$c_r:\mathds{A}^{\mathds{Q}_{>0}}_{\mathds{K}}\longrightarrow \mathds{A}_{\mathds{K}}(r):=\Spec \mathds{K}[X_r]$$ 
qui sont telles que les orbites sous $G_{\mathds{K}}$ de polynômes $\omega\in u^{-1}\overline{\mathds{K}}[u^{-1}]$ sont en bijection avec les points fermés de $\mathds{A}^{\mathds{Q}_{>0}}_{\mathds{K}}$ envoyés sur $0$ par tous les $c_r$ sauf un nombre fini de $c_{k/m}, k\in \mathds{Z}_{>0}$. \\ \indent
Pour tout $r\in\mathds{Q}_{>0}$, on définit $\Omega_{r}(\mathcal{M})$ comme le fermé de $\mathds{A}^{\mathds{Q}_{>0}}_{\mathds{K}}$ constitué des telles orbites de polynômes  $\omega\in u^{-1}\overline{\mathds{K}}[u^{-1}]$ de degré $r$ par rapport à la variable $1/x$ apportant une contribution non-nulle à \eqref{LT}. Notons enfin $\Omega_{<r}(\mathcal{M})$ pour $\sqcup_{r'<r} \Omega_{r'}(\mathcal{M})$. \\ \indent
On suppose $r>1$ et soit $\omega \in \Omega_{r-1}(\mathcal{M})$. On définit par $[\omega]$ l'image de $\omega$ par la composée
\[
\xymatrixcolsep{3pc}
\xymatrix{
\mathds{A}^{\mathds{Q}_{>0}}_{\mathds{K}}\ar[r]^-{c_{r-1}} &  \mathds{A}_{\mathds{K}}(r-1)\ar[r]^-{(1-r)\times} & \mathbf{V}(\mathds{K}\cdot y^{\vee })\simeq T_{r} ^{\vee }
},
\]
où  $(1-r)\times$ est le morphisme de schéma déduit du morphisme de $\mathds{K}$-algèbre associant $(1-r)X_{r-1}$ à $y^{\vee }$. Le point $[\omega]$ est un point fermé de $T_{r}^{\vee }$. On montre aisément le 
\begin{lemme}
Le point $[\omega]$ est indépendant du choix des uniformisantes $x$ et $t$.
\end{lemme}
Notons $\psi_{\pi}$ le foncteur des cycles proches\footnote{Pour une définition précise, voir \ref{sectioncycleproche}.} par rapport à $\pi$ pour les modules holonomes sur $S_{1,n}(D_{k})$. Le but de ce texte est de démontrer le
\begin{theorem} \label{théorème}
On suppose que $r>1$. Alors, le $\mathcal{D}_{T_{r}}$-module $\psi_{\pi} H_{n,k}(\mathcal{M})$ ne dépend de $n$ et $k$ que par l'intermédiaire de $r$, et avec les notations de 1.4, on a la formule
\begin{equation}\label{formule}
\mathfrak{F}\psi_{\pi} H_{n,k}(\mathcal{M})= 
\delta_{0}^{n_{<r-1}^{2}(\mathcal{M})}\oplus\bigoplus_{\omega \in \Omega_{r-1}(\mathcal{M})}\delta_{[\omega]}^{[\mathds{K}(\omega):\mathds{K}([\omega])]n_{\omega}^{2}},
\end{equation}
où $n_{<r-1}^{2}(\mathcal{M})$ est l'entier $\displaystyle{\sum_{\omega \in \Omega_{<r-1}(\mathcal{M})}[\mathds{K}(\omega):\mathds{K}]n_{\omega}^{2}}$.
\end{theorem}

\subsection{}
Soit $R$ un anneau de valuation discrète complet d'égale caracté\-ristique $p$, d'idéal maximal $\mathfrak{M}$ et de corps résiduel $F$, supposé de type fini sur un corps parfait. On note $K$ le corps de fraction de $R$. Soit $S= \Spec R$ le trait complet associé à $R$ et $\eta_{S}$ son point générique. On se donne un entier $n$ multiple de $p$, un caractère $\chi \in H^{1}(K,\mathds{Z}/n\mathds{Z})$ et pour un nombre premier $\ell \neq p$, on fixe une injection $\mathds{Z}/n\mathds{Z}\rightarrow  \overline{\mathds{F}}_{\ell}^{\times}$. On note encore $\chi : G_{K} \rightarrow \overline{\mathds{F}}_{\ell}^{\times}$ le caractère induit, et $\mathcal{F}$ le $\overline{\mathds{F}}_{\ell}$-faisceau étale associé sur $\eta_{S}$.\\ \indent 
Si le conducteur de Swan $\sw(\chi)$ de $\mathcal{F}$ vérifie $\sw(\chi)>1$, Abbes et Saito démontrent \cite[9.10]{AS} que le support de $\mathfrak{F}\psi H_{1, \sw(\chi)+1}(\mathcal{F})$ est réduit à la forme différentielle tordue 
\[
\xymatrix{ 
\rsw(\chi):F \ar[r] & \Omega_{R}^{1}\otimes_{R}(\mathfrak{M}^{-\sw(\chi)-1}/\mathfrak{M}^{-\sw(\chi)})
}
\]
donnée par la théorie de la ramification des caractères d'Artin-Schreier-Witt de Kato \cite[10]{AS}. Le théorème \ref{théorème} pour $\mathcal{M}$ de type exponentiel est l'analogue de \cite[9.10]{AS} pour $F$ parfait. \\ \indent
Quant à la finitude du support de $\mathfrak{F}\psi_{\pi} H_{n,k}(\mathcal{M})$ en général (et le fait que celui-ci ne rencontre pas l'origine lorsque $\mathcal{M}$ est purement de pente $r'>0$ et $r=r'+1$ dans \ref{théorème}), il s'agit de l'analogue de \cite[9.15]{AS}.
\subsection{}
Puisque la construction fait aussi sens lorsque $k\leq n$, on peut se demander ce qu'elle donne  dans ce cas. On montre en \ref{nullité} qu'il n'y a pas grand chose à en attendre, puisque dans le cas particulier le plus simple où $\mathcal{M}$ est décomposé sans partie régulière, on a toujours $\psi_{\pi} H_{n,k}(\mathcal{M})\simeq 0$.

\section{Quelques lemmes sur les cycles proches}
\subsection{Le cas complexe. Généralités et exemples}\label{sectioncycleproche}
Pour les références historiques concernant les cycles proches pour les $\mathcal{D}$-modules, on pourra consulter \cite{Kashpsi} et \cite{Malpsi}. Comme référence de travail, on utilisera \cite{MM}. \\ \indent
Dans toute cette section, $X$ désigne une variété algébrique complexe lisse, $f:X\rightarrow \mathds{A}^{1}_{\mathds{C}}$ un morphisme lisse de fibre spéciale $Y=f^{-1}(0)$, $i:Y \longrightarrow X$ l'inclusion de $Y$ dans $X$ et $\mathcal{I}$ l'idéal de définition de $Y$. Soit
$$
V_{k}(\mathcal{D}_{X})=\{P \in \mathcal{D}_{X},  P(\mathcal{I}^{l})\subset \mathcal{I}^{l-k} \quad \forall l \in \mathds{Z}\}
$$
la $V$-filtration de $\mathcal{D}_{X}$, et soit $\mathcal{M}$ un $\mathcal{D}_{X}$-module spécialisable le long de $Y$ (par exemple un module holonome). Alors on dispose pour toute $V$-filtration $U$ localement image à décalage près de $V_{\textbf{.}}(\mathcal{D}_{X})^{p}$ par une surjection locale $\mathcal{D}_{X}^{p}\longrightarrow \mathcal{M}\longrightarrow 0$ (c'est la propriété de bonté d'une $V$-filtration) d'un unique polynôme unitaire $b_{U}$ vérifiant pour tout $k \in \mathds{Z}$ 
$$
b_{U}(t\partial_{t}+k)U_{k} \in U_{k-1} ,
$$
avec $t$ équation locale de $Y$. On dit que $b_{U}$ est le \textit{polynôme de Bernstein de $(U_{k})_{k \in \mathds{Z}}$}. Il est indépendant du choix de l'équation locale de $Y$. Puisqu'un sous-module d'un module spécialisable est encore spécialisable, on peut définir pour $m \in \mathcal{M}$ le polynôme de Bernstein de $m$ comme le polynôme de Bernstein de la bonne $V$-filtration $V_{\textbf{.}}(\mathcal{D}_{X})m$ sur $\mathcal{D}_{X}m$. On notera $b_{m}$ ce polynôme, et $\ord_{Y}(m)$ l'ensemble de ses racines. \\ \indent
Soit $\geq$ l'ordre l'exicographique sur $\mathds{C}\simeq \mathds{R}+ i\mathds{R}$. Pour $a \in \mathds{C}$, on définit 
$$
V_{a}(\mathcal{M})=\{m \in \mathcal{M},  \ord_{Y}(m) \subset \{ \alpha \in \mathds{C}, \alpha \geq -a-1 \}\}
$$
et 
$$
V_{<a}(\mathcal{M})=\{m \in \mathcal{M},  \ord_{Y}(m) \subset \{ \alpha \in \mathds{C}, \alpha > -a-1 \}\}.
$$
D'après \cite[4.3-5]{MM}, $(V_{a+k}(\mathcal{M}))_{k \in \mathds{Z}}$ (resp. $(V_{<a+k}(\mathcal{M}))_{k \in \mathds{Z}}$ est l'unique bonne $V$-filtration de $\mathcal{M}$ dont les racines du polynôme de Bernstein sont dans l'intervalle $[- a-1, -a[$ (resp. $]- a-1, -a]$). Si $\psi_{f,a}\mathcal{M}$ désigne la quotient $V_{a}(\mathcal{M})/V_{<a}(\mathcal{M})$, on pose 
\begin{equation} \label{psi}
\psi_{f}\mathcal{M}:= \displaystyle{\bigoplus_{-1 \leq a<0}}\psi_{f,a}\mathcal{M}.
\end{equation}
\begin{proposition}\label{MSaitopsi}
Soit $\mathcal{E}$ une connexion algébrique sur $X$ et $\mathcal{M}$ un $\mathcal{D}_{X}$-module holonome. Alors, on a une identification canonique $\psi_{f} (\mathcal{E}\otimes \mathcal{M}) \simeq i^{+}\mathcal{E}\otimes \psi_{f} \mathcal{M}$.
\end{proposition}
\begin{proof}
Soit $V$ une bonne filtration sur $\mathcal{M}$. Pour $k\in \mathds{Z}$, on pose
$$
U_{k}=\mathcal{E}\otimes V_{k}.
$$
\indent
Montrons qu'il s'agit d'une bonne $V$-filtration de $\mathcal{E}\otimes \mathcal{M}$. On va pour cela utiliser le critère \cite[4.1–9]{MM}. \\ \indent
Pour $k\in \mathds{Z}$, il faut commencer par montrer la $V_{0}(\mathcal{D}_{X})$-cohérence de $U_{k}$. Puisque $V_{0}(\mathcal{D}_{X})$ est un faisceau d'anneaux localement noethérien et cohérent \cite[4.1-5]{MM}, il suffit de montrer la finitude locale de $U_{k}$ sur $V_{0}(\mathcal{D}_{X})$. Soit $m_{1}, \dots  ,m_{n}$ un système de $V_{0}(\mathcal{D}_{X})$-générateurs locaux de $V_{k}$ et $e_{1}, \dots  ,e_{n}$ un système de $\mathcal{O}_{X}$-générateurs locaux de $\mathcal{E}$. On va montrer que les $e_i\otimes m_j$ forment un système de $V_{0}(\mathcal{D}_{X})$-générateurs locaux de $U_{k}$.
On se donne $e \in \mathcal{E}$, $f_{1}, \dots  ,f_{n}$ les coefficients de $e$ dans la base des $(e_{i})$, et $m\in V_{k}$. Pour $P \in V_{0}(\mathcal{D}_{X})$, on désigne par $d(P)$ l'ordre de $P$. On pose alors
$$
d_{m}=\Min\{ \Max(d(P_{j})), m=\sum P_{j}m_{j} \text{\hspace{0.1em} avec \hspace{0.05em} }     P_{j}\in V_{0}(\mathcal{D}_{X}) \}.
$$
Il faut montrer que 
\begin{equation}\label{cohérence}
e\otimes m \in \sum V_{0}(\mathcal{D}_{X})\cdot (e_{i}\otimes m_{j}).
\end{equation}
On raisonne par récurrence sur $d_{m}$, le cas $d_{m}=0$ découlant du fait que le produit tensoriel envisagé est pris sur $\mathcal{O}_{X}$. Si $d_{m}>0$, on choisit des opérateurs $P_{j}$ qui réalisent $d_{m}$ et on écrit
$$
\sum P_{j}(e\otimes m_{j})=e\otimes m+\sum Q_{ij}e\otimes R_{ij}m_{j}
$$
avec $d_{R_{ij}m_{j}}<d_{m}$, de sorte que l'hypothèse de récurrence s'applique à  $Q_{ij}e\otimes R_{ij}m_{j}$. Puisque 
$$\sum P_{j}(e\otimes m_{j})=\sum P_{j}f_{i}(e_{i}\otimes m_{j}) \in \sum V_{0}(\mathcal{D}_{X})\cdot (e_{i}\otimes m_{j}),
$$
on en déduit que \eqref{cohérence} est vraie, d'où la  $V_{0}(\mathcal{D}_{X})$-cohérence de $U_{k}$. 
\\ \indent
Soit $k_{0}\in \mathds{N}$ tel que pour tout $k\in \mathds{N}$
\begin{equation}\label{identité}
V_{k_{0}+k}=V_{k}(\mathcal{D}_{X})V_{k_{0}} \text{\quad et \quad} V_{-k_{0}-k}=V_{-k}(\mathcal{D}_{X})V_{-k_{0}}.
\end{equation}
Montrons que $U$ vérifie les identités analogues. Le cas $k=0$ étant immédiat car $V_{0}(\mathcal{D}_{X})$ contient la fonction unité. On peut donc supposer $k>0$. Il suffit alors de démontrer
\begin{equation}\label{identitéanalogues}
U_{k_{0}+k}=U_{k_{0}+k-1}+\partial_{t} U_{k_{0}+k-1}\text{\quad et \quad} U_{-k_{0}-k}=tU_{-k_{0}-k+1}.
\end{equation}
Seules les inclusions directes posent a priori problème. La seconde relation de \eqref{identitéanalogues} découle immédiatement de \eqref{identité} du fait que le produit tensoriel envisagé est pris sur $\mathcal{O}_{X}$. Prouvons la première relation. Soit $e \in \mathcal{E}$ et $m\in V_{k_{0}+k}$. On choisit $m_{1},m_{2} \in V_{k_{0}+k-1}$ tels que $m=m_{1}+\partial_{t}m_{2}$. Alors
$$
e\otimes m=e\otimes m_{1}+\partial_{t}(e\otimes m_{2})-(\partial_{t}e)\otimes m_{2} \in U_{k_{0}+k-1}+\partial_{t} U_{k_{0}+k-1},
$$
d'où \eqref{identitéanalogues}, et par suite $U$ est une bonne $V$-filtration. \\ \indent
En particulier pour $a\in \mathds{C}$, $U_{k}=\mathcal{E}\otimes V_{a+k}(\mathcal{M})$ définit une bonne filtration de $\mathcal{E}\otimes \mathcal{M}$. Pour $e \in \mathcal{E}$ et $m \in V_{a+k}(\mathcal{M})$, on a par lissité de $\mathcal{E}$
\begin{equation*}
\begin{split}
b_{V_{a+}(\mathcal{M})}(t\partial_{t}+k)(e \otimes m)&\in e \otimes b_{V_{a+}(\mathcal{M})}(t\partial_{t}+k)m+U_{k-1} \subset U_{k-1}.                                    
\end{split}
\end{equation*}
On en déduit que $b_{U}$ divise $b_{V_{a+}(\mathcal{M})}$, donc $V_{a}(\mathcal{E}\otimes \mathcal{M})=\mathcal{E}\otimes V_{a}(\mathcal{M})$. De même $V_{<a}(\mathcal{E}\otimes \mathcal{M})=\mathcal{E}\otimes V_{<a}(\mathcal{M})$ et \ref{MSaitopsi} découle alors de la $\mathcal{O}_{X}$-platitude de $\mathcal{E}$.

\end{proof}
Si $Y$ est non lisse, le formalisme précédent ne s'applique pas tel quel. On peut néanmoins toujours définir des cycles proches dans ce cas en plongeant $X$ dans $X\times \mathds{A}^{1}_{\mathds{C}}$ via l'application graphe de $f$ notée $\Gamma(f)$, puis en prenant les cycles proches suivant la projection par rapport au second facteur. On obtient alors un $\mathcal{D}$-module à support dans $X$ et dans le cas où $Y$ est lisse on retrouve bien la définition initiale. \\ \indent
On peut aussi définir le foncteur $\psi$ lorsque le corps de base est un corps $\mathds{K}$ de caractéristique $0$ quelconque. Soit $X$ une variété lisse sur $\mathds{K}$ et soit $Y$ une hypersurface lisse de $X$ donnée comme lieu des zéros d'une fonction $f$.
Soit $\mathcal{M}$ un $\mathcal{D}_{X}$-module holonome, et $\sigma : \overline{\mathds{K}}/\mathds{Z} \longrightarrow \overline{\mathds{K}}$ une section de la projection $\overline{\mathds{K}}\longrightarrow  \overline{\mathds{K}}/\mathds{Z}$ telle que la classe de $0$ soit envoyée sur $1$. En mimant les propositions 4.2-6 et 4.3-5 de \cite{MM}, on obtient le
\begin{lemme}\label{second point de vue}
Il existe une unique bonne $V$-filtration $V^{\sigma}(\mathcal{M})$ de $\mathcal{M}$ dont les racines du polynôme de Bernstein sont dans l'image de $\sigma$.
\end{lemme}
Suivant \cite[4.3.9]{MM}, on est amené à définir
$$
\psi_{f}\mathcal{M}:=
V_{-1}(\mathcal{M})/V_{-2}(\mathcal{M}).
$$
Cette définition est indépendante du choix de $\sigma$ à isomorphisme non canonique près. 
Comme application immédiate de \ref{second point de vue}, on observe que le foncteur $\psi_{f}$ commute à l'extension des scalaires.
\subsection{Quelques compatibilités}
Tout comme dans la situation topologique, les cycles proches sont compatibles au changement de trait. C'est l'objet de la
\begin{proposition}\label{t^{r}}
Soit $n\geq 1$ un entier et $\gamma_n:\mathds{A}^{1}_{\mathds{C}} \longrightarrow \mathds{A}^{1}_{\mathds{C}}$ le morphisme d'élévation à la puissance $n$. Soit $f^{\prime}:X^{\prime}\longrightarrow \mathds{A}^{1}_{\mathds{C}}$ rendant cartésien le diagramme 
\[
\xymatrix{ 
X^{\prime} \ar[r]^{p} \ar[d]_-{f^{\prime}} &  X \ar[d]^-{f} \\
\mathds{A}^{1}_{\mathds{C}}\ar[r]_{\gamma_n}      &   \mathds{A}^{1}_{\mathds{C}}
}
\]
Alors, on a  une identification canonique de $\mathcal{D}_{Y}$-modules
$$
\psi_{f^{\prime}}p^{+}\mathcal{M} \simeq \psi_{f}\mathcal{M}
$$
\end{proposition}
\noindent
Pour la preuve de ceci, voir \cite[2.3.3]{WTSabbah}. Selon \cite[14.10]{Stokes}, les cycles proches sont aussi compatibles à la formalisation le long de $Y$, à savoir qu'on a la
\begin{proposition}\label{formalisation}
Soit $\hat{f}:\hat{X}\rightarrow \hat{\mathds{A}^{1}_{\mathds{C}}}$ la formalisation de $f$ le long de $Y$. Alors, pour tout $\mathcal{D}_{X}$-module spécialisable $\mathcal{M}$, le $\mathcal{D}_{\hat{X}}$-module $\hat{\mathcal{M}}$ est spécialisable et on a une identification canonique $\psi_{f}\mathcal{M} \simeq \psi_{\hat{f}}\hat{\mathcal{M}}$. 
\end{proposition}
\noindent
Si on se donne $n>0$, $\psi_{f}$ est relié suivant  \cite[3.3.13]{PTM} à $\psi_{f^{n}}$ de la façon suivante
\begin{proposition}\label{t^{r}2}
Pour tout $a \in \mathds{C}$, on a une identification canonique 
\[
\xymatrix{ 
\psi_{f^{n},a}\mathcal{M} \ar[r]^-{\sim}   & i_{+}\psi_{f,na}\mathcal{M}
}.
\]
\end{proposition}

\subsection{Quelques critères d'annulation}
On se donne un morphisme propre $p:X\longrightarrow Y$ entre variétés algébriques lisses, et $f:Y\longrightarrow \mathds{A}^{1}_{\mathds{C}}$ une fonction sur $Y$. On note $Z:=f^{-1}(0)$ et on suppose que $p$ est étale au-dessus de $U=Y\setminus Z$. On se donne enfin un $\mathcal{D}_{Y}$-module holonome $\mathcal{M}$ localisé le long de $Z$, à savoir $\mathcal{M} \simeq \mathcal{M}(*Z)$. \\ \indent
Puisque $p$ est étale au-dessus de $U$, le module $\mathcal{H}^{i}p^{+}\mathcal{M}$ est à support dans $p^{-1}(Z)$ pour $i>0$, et ainsi par \cite[4.1]{MT} on a pour $i>0$
$$
\mathcal{H}^{i}p^{+}\mathcal{M} \simeq \mathcal{H}^{i}p^{+}(\mathcal{M}(\ast Z))\simeq \mathcal{H}^{i}(p^{+}\mathcal{M})(\ast p^{-1}(Z))\simeq (\mathcal{H}^{i}p^{+}\mathcal{M})(\ast p^{-1}(Z))\simeq 0
$$ 
Donc $p^{+}\mathcal{M}$ peut être considéré comme un objet de la catégorie des $\mathcal{D}_{X}$-modules, ce qui sera implicitement fait dans la suite. \\ \indent
Toujours puisque $p$ est étale au-dessus de $U$, le module $\mathcal{H}^{i}p_{+}p^{+}\mathcal{M}$ est à support l'hypersurface $Z$ pour $i>0$. Or la relation $\mathcal{M} \simeq \mathcal{M}(*Z)$ donne
$$
p_{+}p^{+}\mathcal{M}\simeq p_{+}((p^{+}\mathcal{M})(\ast p^{-1}(Z)))\simeq (p_{+}p^{+}\mathcal{M})(\ast Z)
$$
où la dernière identification provient de \cite[3.6-4]{Mehbsmf}. On en déduit\footnote{Cette vérification est nécessaire si l'on souhaite utiliser le formalisme des cycles proches exposé ici. Pour un formalisme valable dans un cadre dérivé, on renvoie à \cite{LaurMal}.} que le module $p_{+}p^{+}\mathcal{M}$ est concentré en degré $0$.
\begin{lemme}  \label{groupe fini}
Le module $\psi_{f}  \mathcal{M}$ est nul dès que le module $\psi_{f \circ p}p^{+}\mathcal{M}$ est nul.
\end{lemme}
\begin{proof}
Le morphisme d'adjonction $\mathcal{M} \rightarrow p_{+}p^{+}\mathcal{M}$ est injectif. Par exacti\-tude des cycles proches, $\psi_{f}  \mathcal{M}$ est un sous-objet de $\psi_{f }p_{+}p^{+}\mathcal{M}$, et on conclut à l'aide de la commutation des cycles proches avec l'image directe propre \cite[4.8.1]{MS}.
\end{proof}

Les lemmes d'annulation qui suivent apparaissent déjà dans la littérature  \cite[14.22, 14.26]{Stokes}. On rappelle ici la preuve de \ref{annulation 3} en appendice pour la commodité du lecteur, et on donne une autre preuve de \ref{annulation 2} à l'aide de \ref{groupe fini}.  \\ \indent
Soit $U$ un ouvert de $\mathds{A}_{\mathds{C}}^{2}$ contenant l'origine et $f=f(t,y)$ une fonction régulière sur $U$. Soit $\mathcal{R}$ une connexion sur $U$ méromorphe à singularité régulière le long de $t=0$.
\begin{lemme}  \label{annulation 1}
On suppose que $f(0,0)\neq 0$ ou que $f(0,y)$ admet un zéro simple en l'origine. Alors si $k>0$ et $a>0$, on a $\psi_{t^{a}} (\mathcal{E}^{f(t,y)/t^{k}}\otimes \mathcal{R}) \simeq 0$ au voisinage de $0$.
\end{lemme}
\begin{lemme} \label{annulation 2}
On suppose que $f(0,0)\neq 0$ et soit $g=t^{a}y^{b}$ avec $a\neq 0$ et $b\neq 0$. Alors si $k,k'>0$, on a $\psi_{g} (\mathcal{E}^{f(t,y)/t^{k}y^{k'}}\otimes \mathcal{R}) \simeq 0$ au voisinage de $0$.
\end{lemme}

\begin{lemme} \label{annulation 3}
On fait l'hypothèse que $f=t^{l}g(t,y)+y^{m}h(t,y)$ avec $g(0,0)\neq 0$, $h(0,0)\neq 0$, $(l,m) \neq (0,0)$. Soit $g=t^{a}y^{b}$ avec $(a,b)\neq 0$. Alors si $k > l$ et $k'>0$, on a $\psi_{g}(\mathcal{E}^{f(t,y)/t^{k}y^{k'}}\otimes \mathcal{R})  \simeq 0$ au voisinage de $0$.
\end{lemme}

\section{Preuve du théorème} \label{preuve}
\subsection{Réduction au cas où $\mathds{K}=\mathds{C}$}\label{Réduction}
Il va s'agir d'une application du principe de Lefschetz. Soit $\mathds{L}$ une extension de $\mathds{K}$, $\overline{\mathds{L}}$ une clôture algébrique de $\mathds{L}$ et $\overline{\mathds{K}}$ la clôture algébrique de $\mathds{K}$ dans $\overline{\mathds{L}}$. Soit $\mathcal{M}$ un $\mathds{K}((x))$-module différentiel.
\begin{lemme} \label{extension des scalaires}
La formule \eqref{formule} est vraie pour $\mathcal{M}_{\mathds{L}}$ si et seulement si elle est vraie pour $\mathcal{M}$.
\end{lemme}
\begin{proof}
On commence par observer que les manipulations géométriques intervenant dans la construction d'Abbes et Saito \eqref{construction AS} commutent à l'extension des scalaires, de même que les foncteurs $\psi_{\pi}$ et $\mathfrak{F}$.  \\ \indent
Supposons que \eqref{formule} soit vraie pour $\mathcal{M}$ et considérons le diagramme cartésien
\[
\xymatrix{ 
\mathds{A}^{\mathds{Q}_{>0}}_{\mathds{L}} \ar[r] \ar[d]_-{[\quad]_{\mathds{L}}} &  \mathds{A}^{\mathds{Q}_{>0}}_{\mathds{K}}  \ar[d]^-{[\quad]} \\
T_{r,\mathds{L}} ^{\vee}  \ar[r]      &   T_{r} ^{\vee } 
}
\]
On a 
\begin{equation}\label{sur L1}
\mathfrak{F}\psi_{\pi} H_{n,k}(\mathcal{M_{\mathds{L}}})=\delta_{0}^{n_{<r-1}^{2}(\mathcal{M})}\oplus\bigoplus_{\omega \in \Omega_{r-1}(\mathcal{M})}\bigoplus_{P_{\omega}\in [\omega]\times_{\mathds{K}}\mathds{L} }\delta_{P_{\omega}}^{[\mathds{K}(\omega):\mathds{K}([\omega])]n_{\omega}^{2}}. \\
\end{equation}
Soient $\omega \in \Omega_{r-1}(\mathcal{M})$ et $P_{\omega}\in [\omega]\times_{\mathds{K}}\mathds{L}$. En développant les termes de \eqref{sur L1} grâce aux formules
\begin{equation*}\label{sur L2}
\begin{split}
[\mathds{K}(\omega):\mathds{K}([\omega])]&=\displaystyle{\sum_{\omega^{\prime}\in  P_{\omega}\times_{[\omega]}\omega}} [\mathds{L}(\omega^{\prime}):\mathds{L}([\omega^{\prime}]_{\mathds{L}})], \\
[\mathds{K}(\omega):\mathds{K}]&=\displaystyle{\sum_{\omega^{\prime}\in  \omega\times_{\mathds{K}}\mathds{L}}} [\mathds{L}(\omega^{\prime}):\mathds{L}],
\end{split}
\end{equation*}
on  observe que la somme triple que l'on obtient se fait sur $\Omega_{r-1}(\mathcal{M})_{\mathds{L}}=\Omega_{r-1}(\mathcal{M}_{\mathds{L}})$. Vue de cette façon, elle s'explicite en 
$$
\mathfrak{F}\psi_{\pi} H_{n,k}(\mathcal{M}_{\mathds{L}})= 
\delta_{0}^{n_{<r-1}^{2}(\mathcal{M}_{\mathds{L}})}\oplus\bigoplus_{\omega \in \Omega_{r-1}(\mathcal{M}_{\mathds{L}})}\delta_{[\omega]_{\mathds{L}}}^{[\mathds{L}(\omega):\mathds{L}([\omega]_{\mathds{L}})]n_{\omega}^{2}},
$$
qui est exactement la formule \eqref{formule}  pour $\mathcal{M}_{\mathds{L}}$. \\ \indent
Supposons réciproquement que \eqref{formule} est vraie pour $\mathcal{M}_{\mathds{L}}$, à savoir
\begin{equation}\label{un dernier effort}
\mathfrak{F}\psi_{\pi} H_{n,k}(\mathcal{M})_{\mathds{L}} \simeq  
\delta_{0}^{n_{<r-1}^{2}(\mathcal{M}_{\mathds{L}})}\oplus\bigoplus_{\omega \in \Omega_{r-1}(\mathcal{M})_{\mathds{L}}}\delta_{[\omega]_{\mathds{L}}}^{[\mathds{L}(\omega):\mathds{L}([\omega]_{\mathds{L}})]n_{\omega}^{2}}.
\end{equation}
Alors, le support de $\mathfrak{F}\psi_{\pi} H_{n,k}(\mathcal{M})$ est réduit à un nombre fini de points fermés et \ref{extension des scalaires} provient de ce que les multiplicités de ces points sont inchangées par extension des scalaires.
\end{proof}
Soit $\mathcal{N}$  un modèle algébrique de $\mathcal{M}$, c'est-à-dire un $\mathds{K}[x,x^{-1}]$-module différentiel tel que $\mathcal{M}\simeq  \mathds{K}((x))\otimes_{K[x,x^{-1}]}\mathcal{N}$. Un tel modèle existe d'après \cite[2.4.10]{Katz}. \\ \indent
Soit $\mathds{K}^{\prime}/\mathds{Q}$ l'extension de $\mathds{Q}$ engendrée par les coefficients des polynômes de Laurent intervenant dans la matrice de $\partial_{x}$ dans une base choisie de $\mathcal{N}$. Si on note $\mathcal{N}_{\mathds{K}^{\prime}}$ le $\mathds{K}^{\prime}((x))$-module différentiel que ce choix de base définit, le lemme \ref{extension des scalaires} assure qu'il suffit de prouver \ref{théorème} pour 
$\mathcal{N}_{\mathds{K}^{\prime}}$. Puisque le degré de transcendance de $\mathds{K}^{\prime}/\mathds{Q}$ est fini, on peut se donner un plongement  de $\mathds{K}^{\prime}$ dans $\mathds{C}$. Via ce choix de plongement, on se ramène toujours par \ref{extension des scalaires} à démontrer \ref{théorème} pour le module différentiel complexe qui se déduit de $\mathcal{N}_{\mathds{K}^{\prime}}$ par extension des scalaires. 
\begin{center}
\textbf{Dans toute la suite, on supposera que $\mathds{K}=\mathds{C}$.} \\ \indent
\end{center}

Dans la coordonnée $y=\frac{dx}{x^{r}}$ de $T_{r}$, il s'agit donc de démontrer
\begin{equation}\label{ce qu'il faut prouver}
\psi_{\pi} H_{n,k}(\mathcal{M})= 
\mathcal{O}_{T_{r}}^{n_{<r-1}^{2}(\mathcal{M})}\oplus\bigoplus_{\omega \in \Omega_{r-1}(\mathcal{M})}(\mathcal{E}^{(1-r)c_{r-1}(\omega)y})^{n_{\omega}^{2}}.
\end{equation}

\subsection{Réduction au cas où $n$ est un multiple de $m_{\mathcal{M}}$}\label{Réduction au cas où $m$ divise $n$}

Soit $m$ un entier naturel. Le morphisme $\nu_{m}: S_{mn} \rightarrow S_{n}$ d'élévation à la puissance $m$ donne un cube
\[
\xymatrix@-1pc{
S \times \eta_{mn} \ar[rr]^-{\id \times \nu_{m}} \ar[dd] \ar[dr]_-{j_{mn,mk}} && S \times \eta_{n} \ar[dr]^-{j_{n,k}} \ar'[d][dd]  \\
& S_{1,mn}(D_{mk}) \ar[rr]_(0.4){p} \ar[dd]^(0.4){\pi^{\prime}} && S_{1,n}(D_{k}) \ar[dd]^(0.4){\pi} \\
\eta_{mn}\ar'[r][rr]  \ar[dr] && \eta_{n} \ar[rd] \\
& S_{mn} \ar[rr] && S_{n} \\
}
\]
à faces commutatives et cartésiennes, et on a 
$$
\begin{array}{ll}
p^{+}H_{n,k}(\mathcal{M})&\simeq j_{mn,mk+}(\id \times      \nu_{m})^{+}\mathcal{H}om(p^{+}_{2}\gamma^{+}_{n}\mathcal{M}, p^{+}_{1}\mathcal{M})\\ &\simeq  j_{mn,mk+}\mathcal{H}om(p^{+}_{2}\gamma^{+}_{mn}\mathcal{M}, p^{+}_{1}\mathcal{M}) \\   &= H_{mn,mk}(\mathcal{M}).
\end{array}
$$
Donc par \ref{t^{r}}, il vient 
$$
\psi_{\pi}H_{n,k}(\mathcal{M}) \simeq\psi_{\pi^{\prime}}H_{mn,mk}(\mathcal{M})
$$
de sorte qu'il suffit de démontrer \eqref{ce qu'il faut prouver} pour le couple $(mn,mk)$. On peut donc supposer que $n$ est un multiple de $m_{\mathcal{M}}$.
\subsection{Le cas où $n$ est un multiple de $m_{\mathcal{M}}$}\label{Le cas où $m$ divise $n$}
En complétant $S_{1,n}(D_{k})$ le long de la fibre spéciale de $\pi$, définie par l'idéal $(t)$, on obtient d'après \cite[8.12]{Mat} le schéma $\hat{S_{1,n}(D_{k})}$ d'anneau de fonctions
\[
\xymatrix{
\mathds{C}[ y]\llbracket x,t\rrbracket /(x-t^{n}p)    \ar[r]^-{\sim}
& \mathds{C}[ y]\llbracket t\rrbracket 
},
\]
où on a posé $p=1+yt^{k-n}$. D'après \ref{formalisation}, on ne change pas les cycles proches de $H_{n,k}(\mathcal{M})$ en restreignant la situation à $\hat{S_{1,n}(D_{k})}$, ce que l'on fera dans la suite. Or $p$ admet des racines $n$-ième dans l'anneau de fonctions de $\hat{S_{1,n}(D_{k})}$. Notons $z$ la racine de $p$ satisfaisant à $z \equiv 1+yt^{k-n}/n \quad \left(t^{k-n+1}\right)$, et posons $\tau=zt$. \\ \indent
Puisque par \ref{Réduction au cas où $m$ divise $n$} on peut supposer que l'indice de ramification de $\mathcal{M}$ divise $n$, on a 
$$
\gamma^{+}_{n}\mathcal{M} \simeq \displaystyle{\bigoplus_{\Omega(\mathcal{M})}} \mathcal{E}^{\omega(t)} \otimes \mathcal{R}_{\omega(t)}
$$
Par définition de l'anneau de fonctions de $\hat{S_{1,n}(D_{k})}$, la variable $\tau$ est une racine $n$-ième de $x$, de sorte qu'on a aussi
$$
p^{+}_{1}\mathcal{M} \simeq \displaystyle{\bigoplus_{\Omega(\mathcal{M})}} \mathcal{E}^{\omega(\tau)} \otimes \mathcal{R}_{\omega(\tau)}
$$
En remplaçant $\tau$ par $zt$, il vient
$$
\psi_{ \pi } H_{n,k}(\mathcal{M}) \simeq \displaystyle{\bigoplus_{\omega_{1},\omega_{2}\in \Omega(\mathcal{M})}} \psi_{t}(\mathcal{E}^{\omega_{1}(zt)-\omega_{2}(t)}\otimes \mathcal{R}_{\omega_{1}, \omega_{2}}) ,
$$
avec $\mathcal{R}_{\omega_{1}, \omega_{2}}$ régulier le long de la fibre spéciale. Désignons par $H_{n,k}(\omega_{1}, \omega_{2})(\mathcal{R}_{\omega_{1}, \omega_{2}})$ ou même $H_{n,k}(\omega_{1}, \omega_{2})$ quand aucune confusion n'est possible, le terme de cette somme correspondant aux formes $\omega_{1}$ et $\omega_{2}$, et écrivons $\omega_{i}=P_{i}(t)/t^{q_{i}}$ avec $\deg P_{i}<q_{i}$. $P_{i}(0)$ est le coefficient dominant de $\omega_{i}$ pour la variable $1/t$. Soit $n_{\omega}$ le rang de $\mathcal{R}_{\omega}$. La formule \eqref{ce qu'il faut prouver} (et par suite le théorème \ref{théorème}) se déduit des calculs suivants:
\begin{lemme} \label{regulier pas regulier}
Si $\omega\neq 0$, $\psi_{t} H_{n,k}(\omega,0)\simeq \psi_{t} H_{n,k}(0,\omega) \simeq  0$. 
\end{lemme}
\begin{proof} 
Par définition,  $H_{n,k}(\omega,0)\simeq \mathcal{E}^{\omega(zt)}\otimes \mathcal{R}$ et $H_{n,k}(0, \omega)\simeq \mathcal{E}^{\omega(t)}\otimes \mathcal{R}^{\prime}$, avec $\mathcal{R}$ et $\mathcal{R}^{\prime}$ réguliers, donc \ref{regulier pas regulier} est une application immédiate de \ref{annulation 1}.
\end{proof} 
\begin{lemme}
Si $\omega_{1} \neq \omega_{2}$ sont non nulles, alors $\psi_{t} H_{n,k}(\omega_{1}, \omega_{2})\simeq 0. $ 
\end{lemme}
\begin{proof} Si $q_{1} < q_{2}$, 
\begin{equation*}
\begin{split}
t^{q_{2}}(\omega_{1}(zt)-\omega_{2}(t)) &= t^{q_{2}-q_{1}}z^{-q_{1}}P_{1}(zt)-P_{2}(t) \\
                                     &\equiv  -P_{2}(0) \quad (t),
\end{split}
\end{equation*}
de sorte que le lemme \ref{annulation 1} s'applique, et de même si $q_{1} > q_{2}$. \\
Si $q_{1}=q_{2}=q$,
\begin{equation*}
\begin{split}
t^{q}(\omega_{1}(zt)-\omega_{2}(t)) &= z^{-q}P_{1}(zt)-P_{2}(t) \\
                                    &\equiv  P_{1}(t)-P_{2}(t)-qP_{1}(0)yt^{k-n}/n \quad (t^{k-n+1}).
\end{split}
\end{equation*}
La valuation $t$-adique de $P_{1}-P_{2}$ est finie et plus petite que $q-1$, donc quelle que soit la façon dont elle se compare à $k-n$, le lemme \ref{annulation 1} s'applique. 
\end{proof}
Dans les deux lemmes qui suivent, on utilisera l'observation que $z-1=yt^{k-n}/r_{n}$ avec $r_{n}=\frac{z^n-1}{z-1}=\prod_{\zeta \in U_{n}\setminus{\{1\}}}(z-\zeta)$ unité de $\mathds{C}[ y]\llbracket t\rrbracket$.

\begin{lemme} \label{regulier}
$\psi_{t} H_{n,k}(0,0)\simeq \mathcal{O}_{T_{r}}^{n_{0}^{2}}$. 
\end{lemme}
\begin{proof}
On se donne une base de $\mathcal{R}_{0}$ dans laquelle la matrice de $\partial_{t}$ est de la forme $A/t$ avec $A\in GL_{n_{0}}(\mathds{C})$. Alors 
$$
H_{n,k}(0,0)\simeq (\mathcal{O}(\ast T_{r})^{n_{0}^{2}},d+(A \otimes 1-1 \otimes A)dt/t+A\otimes 1 dz/z),
$$
et il suffit donc de démontrer que si $B,C \in GL_{l}(\mathds{C})$ commutent, alors si on pose $$
\mathcal{H}_{B,C} :=(\mathcal{O}(\ast T_{r})^{l},d+Bdt/t+Cdz/z),
$$ 
on a
$$
\psi_{t}\mathcal{H}_{B,C}\simeq \mathcal{O}_{T_{r}}^{l}.
$$
On raisonne par récurrence sur le rang $l$, le cas où $l=1$ découlant à l'aide de \ref{MSaitopsi} d'un calcul immédiat. Supposons $l>1$. Puisque $B$ et $C$ commutent, le choix d'un vecteur propre commun permet de définir une suite exacte du type
\[
\xymatrix{
0 \ar[r] & \mathcal{H}_{\beta,\gamma} \ar[r] & \mathcal{H}_{B,C}  \ar[r] & \mathcal{H}_{B^{\prime},C^{\prime}}  \ar[r] & 0
},
\]
avec $\beta, \gamma \in \mathds{C}$ et $B^{\prime}$ et $C^{\prime}$ des matrices qui commutent. Donc par passage aux cycles proches, l'hypothèse de récurrence donne la suite exacte
$$
\xymatrix{
0 \ar[r] &  \mathcal{O}_{T_{r}} \ar[r] & \psi_{t}\mathcal{H}_{B,C}  \ar[r] &  \mathcal{O}_{T_{r}}^{l-1}  \ar[r] & 0
}.
$$
et \ref{regulier} découle alors de la nullité de $\Ext^{1}(\mathcal{O}_{T_{r}},\mathcal{O}_{T_{r}})$.
\end{proof}
\begin{lemme}
 On suppose $\omega_{1}=\omega_{2}=\omega$ de degré $q$. Alors, 
\begin{enumerate}
\item Si $k-n<q$, $\psi_{t} H_{n,k}(\omega,\omega)\simeq 0$.  \
\item Si $k-n>q$, $\psi_{t} H_{n,k}(\omega,\omega) \simeq \mathcal{O}_{T_{r}}^{n_{\omega}^{2}}$. \
\item Si $k-n=q$, $\psi_{t} H_{n,k}(\omega,\omega) \simeq (\mathcal{E}^{-(k/n-1)P(0)y})^{n_{\omega}^{2}}$.
\end{enumerate}
\end{lemme}   
\begin{proof}
On a 
\begin{equation}
\begin{split}
t^{q}(\omega(zt)-\omega(t)) &= z^{-q}P(zt)-P(t) \\
                            &\equiv (z^{-q}-1)P(t) \quad (t^{k-n+1}) \\
                            &\equiv -r_{q}yt^{k-n}P(t)/r_{n} \quad (t^{k-n+1})  \\
                            &\equiv -r_{q}yt^{k-n}P(0)/r_{n} \quad (t^{k-n+1}),
\end{split}
\end{equation}
de sorte que \ref{annulation 1} assure la nullité de $\psi_{t} H_{n,k}(\omega,\omega)$ dans le cas $k-n<q$.
Si $k-n \geq q$, $\omega_{1}(zt)-\omega_{2}(t)$ est polynomiale,  donc $(2)$ et $(3)$ se déduisent de \ref{MSaitopsi}.
\end{proof}
Ceci achève la preuve de la formule \eqref{ce qu'il faut prouver}.

\appendix
\section{Preuve des lemmes \ref{annulation 2} et \ref{annulation 3}}
\subsection{Preuve de \ref{annulation 2}}
On se ramène à montrer la nullité de $\psi_{t^{a}y^{b}}( \mathcal{E}^{1/t^{k}y^{k'}}\otimes \hat{\mathcal{R}_{0}})$ avec $\hat{\mathcal{R}_{0}}$ de rang $1$ admettant un générateur $m$ qui vérifie $t\partial_{t}m= cm$, $c \in \mathds{C}$ et $\partial_{y}m=0$. Si on définit 
$$
\begin{array}{ccc}
h  :  \mathds{A}^{2}_{\mathds{C}} & \longrightarrow & \mathds{A}^{2}_{\mathds{C}} \\
                          (u,z)  & \mapsto & (u^{b},z^{a}), \\
\end{array}
$$
on sait d'après \ref{groupe fini} qu'il suffit de montrer la nullité du module 
$$\psi_{ (uz)^{ab}}(\mathcal{E}^{1/u^{bk}z^{ak'}}\otimes h^{+}\hat{\mathcal{R}_{0}}),$$ et on est ramené à étudier le cas où $a=b$. Par \ref{t^{r}2}, on peut supposer $a=b=1$. Par définition, $\psi_{ty}=\psi_{v}\circ \Gamma(g)_{+}$ où $\Gamma(g)$ désigne le graphe de $g$ et $v$ est donnée par
\[
\xymatrix{ 
 \mathds{A}^{2}_{\mathds{C}} \ar@{^{(}->}[r]^-{\Gamma(g)} \ar[rd]_-{g}& \mathds{A}^{2}_{\mathds{C}} \times \mathds{A}^{1}_{\mathds{C}}\ar[d]^-{v}   \\
                                                                        &       \mathds{A}^{1}_{\mathds{C}}
}
\] 
Or $s=\text{\og}  e^{1/t^{k}y^{k'}}\text{\fg}\otimes m$ engendre $\mathcal{E}^{1/t^{k}y^{k'}}\otimes \hat{\mathcal{R}_{0}}$ et vérifie $t^{k}y^{k'+1}\partial_{y}s=-k's$. Donc $s'=s \delta(v-ty)$ engendre $\Gamma(g)_{+}\mathcal{E}^{1/t^{k}y^{k'}}\otimes \hat{\mathcal{R}_{0}}$. Or
$$
v\partial_{y}\cdot s'=\partial_{y}v\cdot s'=\partial_{y}\cdot(tys\delta)=ts'+ty(\partial_{y}s)\delta-yt^{2}s\partial_{v}\delta \in V_{-1}s'.
$$
donne par multiplication par $y$
$$
yts'+ty^2(\partial_{y}s)\delta-y^2t^{2}s\partial_{v}\delta \in V_{-1}s'.
$$
Puisque $yts'=vs'\in V_{-1}s'$, on en déduit 
$$
ty^2(\partial_{y}s)\delta-y^2t^{2}s\partial_{v}\delta \in V_{-1}s'.
$$
D'autre part
$$
\partial_{v}v^{2}s'=2vs'+v^{2}\partial_{v}s'=t^{2}y^{2}s\partial_{v}\delta \in V_{-1}s'
$$ 
d'où finalement $ty^2(\partial_{y}s)\delta \in V_{-1}s'$. Puisque $k\geq 1$ et $k'\geq 1$, il fait sens de multiplier cette dernière relation par $t^{k-1}y^{k'-1}$. On obtient alors
$$
t^{k}y^{k'+1}(\partial_{y}s)\delta=-k's\delta=-k's' \in V_{-1}s'
$$
Ainsi, $s'\in V_{-1}s'$ est de polynôme de Bernstein constant, d'où \ref{annulation 2}.
\subsection{Preuve de \ref{annulation 3}}
Les cas $l=0$ ou $m=0$ étant traités par \ref{annulation 2}, on peut raisonner par récurrence sur $(l,m)$ et supposer $l\neq 0$ et $m\neq 0$. Soit $p:\tilde{U} \rightarrow U$ l'éclaté de $U$ en l'origine. Par compatibilité des cycles proches avec les modifications propres \cite[14.12]{Stokes}, on a
$$
\psi_{g}(\mathcal{E}^{f(t,y)/t^{k}y^{k'}}\otimes \mathcal{R})  \simeq 
p_{+}\psi_{g \circ p}(\mathcal{E}^{f(p_{t},p_{y})/p_{t}^{k}p_{y}^{k'}}\otimes p^{+}\mathcal{R}).
$$
Démontrons que $\psi_{g \circ p}(\mathcal{E}^{f(p_{t},p_{y})/p_{t}^{k}p_{y}^{k'}}\otimes p^{+}\mathcal{R})\simeq 0 $ sur le diviseur exceptionnel $E$. \\ \indent
Dans la carte $U_{0}$ de $\tilde{U}$ donnée par $t=u$ et $y=uv$, on a $p(u,v)=(u,uv)$, de sorte que la trace de E sur $U_{0}$ est donnée par $u=0$, et 
$$
f(p_{t},p_{y})/p_{t}^{k}p_{y}^{k'}=(u^{l}g(u,uv)+(uv)^{m}h(u,uv))/u^{k+k'}v^{k'}.
$$
Au voisinage d'un point $v_{0}\neq 0$ de $U_{0}$, \ref{annulation 1} s'applique immédiatement si $l \neq m$. Dans le cas $l=m$, on observe que $\frac{\partial f}{\partial y}(0,v_{0})=mv_{0}^{m-1}h(0,0)\neq 0$ de sorte que \ref{annulation 1} s'applique encore. Au voisinage de l'origine de $U_{0}$, \ref{annulation 2} s'applique si $l<m$, et sinon $m>0$ permet d'utiliser l'hypothèse de récurrence. \\ \indent
Dans la carte $U_{1}$ donnée par les coordonnées $t=u'v'$ et $y=v'$, $p(u',v')=(u'v',v')$, donc $E \cap U_{1}$ est défini par $u'v'=0$, et on a 
$$
f(p_{t},p_{y})/p_{t}^{k}p_{y}^{k'}=
((u'v')^{l}g(u'v',v')+v'^{m}h(u'v',v'))/u'^{k}v'^{k+k'}.
$$
On se place au voisinage de l'origine. Si $l \leq m$, la condition $l>0$ assure que l'hypothèse de récurrence s'applique, et sinon la situation est justiciable de \ref{annulation 2}.
\section{Nullité des cycles proches pour $r\leq 1$ et $\mathcal{M}$ décomposé}
\begin{proposition}\label{nullité}
On suppose $k\leq n$ et $\mathcal{M} \simeq \oplus \mathcal{E}^{\omega} \otimes  \mathcal{R}_{\omega}$ sans partie régulière. Alors $\psi_{\pi} H_{n,k}(\mathcal{M})\simeq 0$.
\end{proposition}
\begin{proof}
On a 
$$
\psi_{ \pi } H_{n,k}(\mathcal{M}) \simeq \displaystyle{\bigoplus_{\omega_{1},\omega_{2}}} \psi_{t}(\mathcal{E}^{\omega_{1}(x)-\omega_{2}(t^{n})}\otimes \mathcal{R}_{\omega_{1}, \omega_{2}}), 
$$
avec $\mathcal{R}_{\omega_{1}, \omega_{2}}$ régulier le long de la fibre spéciale. Désignons par $H_{n,k}(\omega_{1}, \omega_{2})$ le terme de cette somme correspondant aux formes $\omega_{1}$ et $\omega_{2}$, et écrivons $\omega_{i}=P_{i}(x)/x^{q_{i}}$ avec $\deg P_{i}<q_{i}$. $P_{i}(0)$ est le coefficient dominant de $\omega_{i}$ pour la variable $1/x$. $\mathcal{M}$ étant supposé sans partie régulière, $P_{i}(0)\neq 0$. Alors en effectuant le changement de variable $u=t$ et $v=t^{n-k}+y$, on a 
\begin{equation}\label{calcul}
\begin{split}
\omega_{1}(x)-\omega_{2}(t^{n}) &= \frac{P(t^{n}+yt^{k})}{t^{kq_{1}}(t^{n-k}+y)^{q_{1}}}-\frac{P_{2}(t^{n})} {t^{nq_{2}}}                     =\frac{P_{1}(u^{k}v)}{u^{kq_{1}}v^{q_{1}}}-\frac{P_{2}(u^{n})}{u^{nq_{2}}} \\
                                &=\frac{u^{nq_{2}}P_{1}(u^{k}v)-u^{kq_{1}}v^{q_{1}}P_{2}(u^{n})}{u^{kq_{1}+nq_{2}}v^{q_{1}}}                                                                                                                  .
\end{split}
\end{equation}

On se place au voisinage de $v_{0}\neq 0$. Si on suppose $nq_{2}>kq_{1}$, \eqref{calcul} s'écrit $(u^{nq_{2}-kq_{1}}P_{1}(u^{k}v)-v^{q_{1}}P_{2}(u^{n}))/u^{nq_{2}}v^{q_{1}}$. Si $f(u,v)$ désigne la partie non polaire de cette expression, $f(0,v_{0})=-P_{2}(0)/ v_{0}^{q_{1}}$ est non nul, donc \ref{annulation 1} s'applique. On raisonne de même avec le cas $nq_{2}<kq_{1}$ en utilisant $P_{1}(0)\neq 0$. Si $nq_{2}=kq_{1}$, la partie non polaire de \eqref{calcul} est $f(u,v)=(P_{1}(u^{k}v)-v^{q_{1}}P_{2}(u^{n}))/v^{q_{1}}$ et en évaluant en $(0,v)$ on constate que la situation satisfait les hypothèses de \ref{annulation 1}. \\ \indent
On se place au voisinage de $v_{0}=0$. Si on suppose $nq_{2}\leq kq_{1}$, \eqref{calcul} devient $(P_{1}(u^{k}v)-u^{kq_{1}-nq_{2}}v^{q_{1}}P_{2}(u^{n}))/u^{kq_{1}}v^{q_{1}}$. Le numérateur  vaut $P_{1}(0)\neq 0$, donc \ref{annulation 2} s'applique. Sinon,  \eqref{calcul}  est $(u^{nq_{2}-kq_{1}}P_{1}(u^{k}v)-v^{q_{1}}P_{2}(u^{n}))/u^{nq_{2}}v^{q_{1}}$ et on est dans le cadre de \ref{annulation 3}.                                                                                                           
\end{proof}

\bibliographystyle{amsalpha}
\bibliography{ASv3}

\end{document}